\documentclass[10pt]{article}

\usepackage{amsmath,amsthm,amsfonts,amssymb,latexsym,graphicx,epigraph,accents}
\usepackage[authoryear]{natbib}
\bibpunct{(}{)}{;}{a}{,}{,}
\usepackage{calc}
\usepackage{enumerate}
\usepackage{etoolbox}
\usepackage{algorithm}
\usepackage[noend]{algpseudocode}
\algrenewcommand\algorithmicthen{\relax}
\algrenewcommand\algorithmicdo{\relax}
\usepackage[colorlinks=true,citecolor=blue,pdfpagemode=UseNone,pdfstartview=FitH]{hyperref}

\allowdisplaybreaks[4]

\newcommand{\I}{^{\textnormal{\tiny I}}}
\newcommand{\II}{^{\textnormal{\tiny II}}}
\newcommand{\m}{^{\textnormal{\footnotesize-}}}

\newcommand{\dd}{\mathrm{d}}

\newcommand{\R}{\mathbb{R}}

\newcommand{\K}{\mathcal{K}}

\theoremstyle{plain}
\newtheorem{theorem}{Theorem}

\newtheorem{lemma}[theorem]{Lemma}

\theoremstyle{definition}
\newtheorem{protocol}[theorem]{Protocol}

\newtheorem{example}[theorem]{Example}

\theoremstyle{remark}
\newtheorem{remark}[theorem]{Remark}

\newlength{\IndentI}
\newlength{\IndentII}
\newlength{\IndentIII}
\newlength{\IndentIV}
\newlength{\IndentV}
\newcommand{\indentI}{\noindent\hspace*{\IndentI}}
\newcommand{\indentII}{\noindent\hspace*{\IndentII}}
\newcommand{\indentIII}{\noindent\hspace*{\IndentIII}}
\newcommand{\indentIV}{\noindent\hspace*{\IndentIV}}
\newcommand{\indentV}{\noindent\hspace*{\IndentV}}

\newcommand{\abstr}{%
  This paper establishes the asymptotic uniqueness of long-term probability forecasts
  in the following form.
  Consider two forecasters who repeatedly issue
  probability forecasts for the infinite future.
  The main result of the paper says that either at least one of the two forecasters will be discredited
  or their forecasts will converge in total variation.
  This can be regarded as a game-theoretic version of the classical Blackwell--Dubins result
  getting rid of some of its limitations.
  This result is further strengthened
  along the lines of Richard Jeffrey's radical probabilism.}

\begin{document}
\title{Asymptotic uniqueness in long-term prediction}
\author{Vladimir Vovk}

\maketitle
\begin{abstract}
  \smallskip
  \abstr

  The version of this paper at \url{http://probabilityandfinance.com} (Working Paper 66)
  is updated most often.
\end{abstract}

\setlength\epigraphwidth{.85\textwidth}
\epigraph{\dots scientific disagreements tend to disappear\dots
  when new data accumulate\dots.}%
  {\citet[p.~673]{Jeffreys:1938}; also in \citet[Sect.~1.9]{Jeffreys:1961}}

\section{Introduction}

This paper belongs to the general area of probabilistic prediction,
and we will be interested in ways of testing long-term predictions
and the asymptotic uniqueness of successful long-term predictions.
The long-term nature of such predictions makes our results
quite different from the usual results in conformal prediction
\citep{Vovk/etal:2022book,Angelopoulos/Bates:2023},
prediction with expert advice
\citep{Cesa-Bianchi/Lugosi:2006,Vovk:1998},
and game-theoretic probability \citep{Shafer/Vovk:2019},
which are usually concerned with one-step-ahead prediction.
In this paper ``prediction'' and ``forecast'' are used interchangeably.

To demonstrate the asymptotic uniqueness of successful long-term predictions,
we consider two forecasters issuing such predictions.
The asymptotic uniqueness holds
if either the forecasts that they issue eventually become almost indistinguishable
or we are able to demonstrate that at least one of the forecasters is inadequate.
This paper establishes new results of this kind and reviews related known results.

\citet[Sect.~2]{Blackwell/Dubins:1962} prove a classical result
about the asymptotic uniqueness of long-term predictions.
They consider two probability measures that agree in a certain sense
(namely, one of them is absolutely continuous w.r.\ to the other).
They then show that the predictions output by the two probability measures
for the infinite future converge in total variation almost surely as time progresses.
We will remove unnecessary restrictions
and generalize this result.

The general phenomenon of convergence of opinions
(i.e., forecasts for the future)
for adequate forecasters will be referred to as \emph{Jeffreys's law},
although originally this expression was used by \citet[Sect.~5.2]{Dawid:1984}
for his result about convergence of one-step-ahead predictions:
\begin{quote}
  I shall call this finding ``Jeffreys's Law'',
  after an admittedly distorted interpretation of Jeffreys (1938):
  ``When a law has been applied to a large body of data
  without any systematic discrepancy being detected\dots\
  the probability of a further inference from the law
  approaches certainty whether the law is true or not.''
\end{quote}
Another quote from the same paper \citep{Jeffreys:1938} is given as the epigraph to this paper.

We start in Sect.~\ref{sec:basic} from discussing ways of testing long-term predictions
and applying those ways to deriving our first version of Jeffreys's law.
This version is stated in terms of a testing protocol
involving three players: Forecaster~I, Forecaster~II,
and a new player, Sceptic, who performs testing.
We will construct a strategy for Sceptic
that discredits at least one of the two forecasters
if their opinions do not converge.
This removes limitations of Blackwell and Dubins's result,
as discussed in Sect.~\ref{sec:comparison}.

One respect in which the result of Sect.~\ref{sec:basic}
(and the Blackwell--Dubins original result)
is restrictive
is that it assumes that we observe data with certainty.
Section~\ref{sec:radical} extends the result of Sect.~\ref{sec:basic}
to Jeffrey's picture of radical probabilism,
in which no evidence is certain.

Proofs of the results of Sects~\ref{sec:basic}--\ref{sec:radical}
are postponed to Sect.~\ref{sec:proofs}.
All our proofs will be constructive
and will exhibit simple strategies for Sceptic
that enforce convergence of opinions.

Section~\ref{sec:comparison} discusses related results in literature
starting from the Blackwell--Dubins result (Sect.~\ref{subsec:BD})
and then going on to results about one-step-ahead prediction (Sect.~\ref{subsec:1-step}).
The most conspicuous difference between the two kinds of results,
those for long-term and one-step-ahead prediction,
is that the latter can be both asymptotic and small-sample,
while the former are invariably asymptotic (to the best of my knowledge);
making them non-asymptotic looks to me an important direction of further research.
Section~\ref{sec:conclusion} concludes
and lists some other directions of further research.

This paper and its predecessor \citep{Vovk:arXiv2308}
were motivated by a conversation with A. Philip Dawid \citep{Vovk/Shafer:2023}.
Among topics of the conversation were one-step-ahead vs multi-step prediction
\citep[Sect.~7]{Vovk/Shafer:2023}
and the Blackwell--Dubins theorem
\citep[Sect.~7 of the arXiv version]{Vovk/Shafer:2023}.

Two people with similar surnames will play key roles in this paper,
Jeffreys and Jeffrey.
Harold Jeffreys (1891--1989) was a contemporary of Ronald Fisher
who spent his professional life in Cambridge, England.
His main speciality was geophysics,
but he was also one of the founders of Bayesian statistics.
Richard Jeffrey (1926--2002) was an influential American philosopher.

\section{Jeffreys's law}
\label{sec:basic}

Before we state Jeffreys's law
we need to discuss ways of testing probability forecasts.
Let us fix a finite \emph{observation space} $\mathbf{Y}$
(equipped with the discrete $\sigma$-algebra).
At each step Reality produces an \emph{observation} $y_n\in\mathbf{Y}$,
$n=1,2,\dots$,
and Forecaster is trying to predict future observations
by issuing a probability forecast $P_n\in\mathfrak{P}(\mathbf{Y}^{\infty})$
($\mathfrak{P}(\mathbf{Y}^{\infty})$ standing
for the family of all probability measures on $\mathbf{Y}^{\infty}$
equipped with the Borel $\sigma$-algebra).
In the following testing protocol,
we let $\mathbf{Y}^*$ stand for the set of all finite sequences of observations,
$\mathbf{Y}^+$ stand for the set of all non-empty finite sequences of observations,
$\lvert x\rvert$ stand for the length of $x\in\mathbf{Y}^*$,
and $[x]$ stand for the set of all infinite continuations
of a finite sequence $x\in\mathbf{Y}^*$
(in other words, $[x]$ is the set of all sequences in $\mathbf{Y}^{\infty}$
that have $x$ as their prefix).

\begin{protocol}\label{prot:basic-1}
  \textbf{Testing protocol:}\\
  \indentI $\K_0 := 1$\\
  \indentI FOR $n=1,2,\dots$:\\
    \indentII Forecaster announces $P_n\in\mathfrak{P}(\mathbf{Y}^{\infty})$\\
    \indentII IF $n>1$:\\
      \indentIII $\K_{n-1} := \K\m_{n-1} + \sum_{x\in\mathbf{Y}^+} f_{n-1}(y_{n-1}x) P_n([x])\\
          \indentV {}- \sum_{x\in\mathbf{Y}^*:\lvert x\rvert>1} f_{n-1}(x) P_{n-1}([x])$
          \hfill\refstepcounter{equation}(\theequation)\label{eq:K-S}\\
    \indentII Sceptic announces $f_n\in\R^{\mathbf{Y}^+}$ such that\\
      \indentIV $f_n(x)=0$ for all but finitely many $x\in\mathbf{Y}^+$\\
    \indentII Reality announces $y_n\in\mathbf{Y}$\\
    \indentII $\K\m_n := \K_{n-1} + f_n(y_n) - \sum_{y\in\mathbf{Y}} f_n(y) P_n([y])$.
    \hfill\refstepcounter{equation}(\theequation)\label{eq:K-R}
\end{protocol}

Protocol~\ref{prot:basic-1} is interpreted in terms of betting,
as described in \citet[Chap.~1]{deFinetti:1937} and,
in our current terminology, \citet[Sect.~3]{Vovk:arXiv2308}.
At each step $n$ Forecaster announces a probability measure $P_n$
for the infinite future $y_n,y_{n+1},\dots$.
The betting interpretation of $P_n$ is that,
for each non-empty finite sequence $x\in\mathbf{Y}^*$,
$P_n([x])$ is the price of a ticket (the \emph{$x$-ticket})
that pays $1_{\{x\subseteq(y_n,y_{n+1},\dots)\}}$
(i.e., it pays 1 if and only if $x$ is a prefix
of the sequence $(y_n,y_{n+1},\dots)$ of the future observations
and pays nothing otherwise).
Forecaster allows his opponent to buy any real number
(positive, negative, or zero; not necessarily integer)
of such tickets.

Testing the forecasts is performed by another player, Sceptic.
At step $n$,
for each non-empty $x\in\mathbf{Y}^*$,
Sceptic announces the number $f_n(x)$ of $x$-tickets
that he chooses to buy at this step.
Therefore, his move is $f_n\in\R^{\mathbf{Y}^+}$,
which means that $f_n:\mathbf{Y}^+\to\R$.
The numbers $f_n(x)$ are allowed to be different from zero
only for finitely many $x$-tickets,
and so the sums $\sum_{x}$ in the protocol are uncontroversial.
At the end of each step Reality announces the actual observation $y_n\in\mathbf{Y}$,
and the $y$-tickets for $y\in\mathbf{Y}$ are cashed in:
the $y_n$-ticket pays $1$, the other $y$-tickets do not pay anything,
and the total cost of all the $y$-tickets is
$\sum_{y\in\mathbf{Y}} f_n(y) P_n([y])$.
The $x$-tickets for longer $x$ are sold at the next step at the new prices
(which accounts for the term
$\sum_{x\in\mathbf{Y}^+} f_{n-1}(y_{n-1}x) P_n([x])$
in Protocol~\ref{prot:basic-1})
and the loss due to their total cost is recorded also at the next step
(which accounts for the term
$\sum_{x\in\mathbf{Y}^*:\lvert x\rvert>1} f_{n-1}(x) P_{n-1}([x])$).

The interpretation of $\K_n$ is that it is Sceptic's capital at time $n$,
but a large $\K_n$ only means that Forecaster's predictions $P_1,P_2,\dots$
have been discredited
(so Sceptic is using ``play money'').
For this interpretation to be valid,
Sceptic is never allowed to go into debt:
as soon as $K_n<0$, the game is stopped
and Sceptic's attempt at discrediting Forecaster fails.

In the case of one-step-ahead prediction we do not need
the capital update \eqref{eq:K-S},
and \eqref{eq:K-R} is sufficient.
On the other hand, for multi-step prediction,
we do not need to have \eqref{eq:K-R}
(which is the standard capital update in game-theoretic probability)
as a separate entry and can merge it into~\eqref{eq:K-S}.
The following protocol is a slightly simplified version of Protocol~\ref{prot:basic-1}.

\begin{protocol}\label{prot:basic-1-simple}
  \textbf{Simplified testing protocol:}\\
  \indentI $\K_0 := 1$\\
  \indentI FOR $n=1,2,\dots$:\\
    \indentII Forecaster announces $P_n\in\mathfrak{P}(\mathbf{Y}^{\infty})$\\
    \indentII IF $n>1$:\\
      \indentIII $\K_{n-1} := \K_{n-2} + \sum_{x\in\mathbf{Y}^*} f_{n-1}(y_{n-1}x) P_n([x])\\
          \indentV {}- \sum_{x\in\mathbf{Y}^+} f_{n-1}(x) P_{n-1}([x])$
          \hfill\refstepcounter{equation}(\theequation)\label{eq:K-merged}\\
    \indentII Sceptic announces $f_n\in\R^{\mathbf{Y}^+}$ such that\\
      \indentIV $f_n(x)=0$ for all but finitely many $x\in\mathbf{Y}^+$\\
    \indentII Reality announces $y_n\in\mathbf{Y}$.
\end{protocol}

\noindent
In Protocol~\ref{prot:basic-1-simple} we merge the steps~\eqref{eq:K-S} and~\eqref{eq:K-R}
of Protocol~\ref{prot:basic-1} into one step~\eqref{eq:K-merged}.

The testing protocol that we use for stating Jeffreys's law
involves one Sceptic playing simultaneously (but separately) against two forecasters
who output probability forecasts $P\I_n,P\II_n\in\mathfrak{P}(\mathbf{Y}^{\infty})$
at steps $n=1,2,\dots$.

\begin{protocol}\label{prot:basic-2}
  \textbf{Double testing protocol:}\\
  \indentI $\K\I_0 = \K\II_0 := 1$\\
  \indentI FOR $n=1,2,\dots$:\\
    \indentII Forecaster I announces $P\I_n\in\mathfrak{P}(\mathbf{Y}^{\infty})$\\
    \indentII Forecaster II announces $P\II_n\in\mathfrak{P}(\mathbf{Y}^{\infty})$\\
    \indentII IF $n>1$:\\
      \indentIII $\K\I_{n-1} := \K\I_{n-2} + \sum_{x\in\mathbf{Y}^*} f\I_{n-1}(y_{n-1}x) P\I_n([x])\\
          \indentV {}- \sum_{x\in\mathbf{Y}^+} f\I_{n-1}(x) P\I_{n-1}([x])$\\
      \indentIII $\K\II_{n-1} := \K\II_{n-2} + \sum_{x\in\mathbf{Y}^*} f\II_{n-1}(y_{n-1}x) P\II_n([x])\\
          \indentV {}- \sum_{x\in\mathbf{Y}^+} f\II_{n-1}(x) P\II_{n-1}([x])$\\
    \indentII Sceptic announces $f\I_n,f\II_n\in\R^{\mathbf{Y}^+}$ such that\\
      \indentIV $f\I_n(x)=f\II_n(x)=0$ for all but finitely many $x\in\mathbf{Y}^+$\\
    \indentII Reality announces $y_n\in\mathbf{Y}$.
\end{protocol}

We will state Jeffrey's law using the total variation distance
\[
  \left\|
    P - Q
  \right\|
  :=
  2
  \sup_{E}
  \left|
    P(E) - Q(E)
  \right|
  \in
  [0,2]
\]
between probability measures $P$ and $Q$ on $\mathbf{Y}^\infty$,
where $E$ ranges over the events in the common domain of $P$ and $Q$.

\begin{theorem}\label{thm:main}
  Sceptic has a strategy in Protocol~\ref{prot:basic-2}
  that guarantees the disjunction of
  \begin{itemize}
  \item
    $\left\|P\I_n-P\II_n\right\|\to0$ as $n\to\infty$,
  \item
    $\K\I_n\to\infty$ as $n\to\infty$,
  \item
    $\K\II_n\to\infty$ as $n\to\infty$.
  \end{itemize}
\end{theorem}

Theorem~\ref{thm:main} can be interpreted as establishing a connection
between the correspondence (see, e.g., \citealt{Popper:1945})
and convergence (see, e.g., \citealt{Peirce:1877})
theories of truth.
The correspondence theory of truth
(of the probability forecasts in this case)
refers to agreement with reality,
and our interpretation of a large $\K_n$ is lack of agreement with reality.
The convergence theory of truth
regards truth as the point at which different opinions converge.
In Peirce's words,
``the settlement of opinion is the sole end of inquiry''
\citep{Peirce:1877}.
According to Theorem~\ref{thm:main},
a version of the correspondence theory implies a version of the convergence theory.

\section{Agnostic probabilism}
\label{sec:radical}

According to the idea of \emph{radical probabilism},
put forward very clearly by \citet{Jeffrey:1968},
empirical evidence is never certain.
In particular, we never learn the outcomes $y_1,y_2,\dots$ for sure.

Jeffrey referred to the opposite point of view,
in which we do observe the true $y_n$ (in our current context),
as dogmatic probabilism.
We would like our mathematical results to cover dogmatic probabilism as a special case,
and the title of this section, \emph{agnostic probabilism},
refers to its prediction protocol eventually disclosing, or never disclosing, the true outcomes $y_n$,
as the case may be for different $n$.
In particular, the following protocol includes Protocol~\ref{prot:basic-2}
as a special case.

\begin{protocol}\label{prot:additive}
  \textbf{Double agnostic testing protocol:}\\
  \indentI $\K\I_0 = \K\II_0 := 1$\\
  \indentI FOR $n=1,2,\dots$:\\
    \indentII Forecaster I announces $P\I_n\in\mathfrak{P}(\mathbf{Y}^{\infty})$\\
    \indentII Forecaster II announces $P\II_n\in\mathfrak{P}(\mathbf{Y}^{\infty})$\\
    \indentII IF $n>1$:\\
      \indentIII $\K\I_{n-1} := \K\I_{n-2} + \sum_{x\in\mathbf{Y}^+}
          f\I_{n-1}(x) (P\I_n([x]) - P\I_{n-1}([x]))$%
          \hfill\refstepcounter{equation}(\theequation)\label{eq:K-I}\\
      \indentIII $\K\II_{n-1} := \K\II_{n-2} + \sum_{x\in\mathbf{Y}^+}
          f\II_{n-1}(x) (P\II_n([x]) - P\II_{n-1}([x]))$\\
    \indentII Sceptic announces $f\I_n,f\II_n\in\R^{\mathbf{Y}^+}$ such that\\
      \indentIV $f\I_n(x)=f\II_n(x)=0$ for all but finitely many $x\in\mathbf{Y}^+$.
\end{protocol}

\noindent
Protocol~\ref{prot:additive} does not include Reality as a separate player.
Her role is played by the two forecasters:
in order to model Reality of Protocol~\ref{prot:basic-2},
they should choose $P\I_n$ and $P\II_n$ concentrated on
$[(y_1,\dots,y_{n-1})]$ for some observations $y_1,\dots,y_{n-1}\in\mathbf{Y}$.
Such a choice already implies a certain agreement between the forecasters;
in particular, $P\I_n-P\II_n\to0$ as $n\to\infty$ \emph{weakly},
meaning that, for any $x\in\mathbf{Y}^*$,
\[
  P\I_n([x]) - P\II_n([x])
  \to
  0
  \quad
  \text{as $n\to\infty$}.
\]
More precisely,
to embed Protocol~\ref{prot:basic-2} into Protocol~\ref{prot:additive},
the two forecasters in the latter should choose the probability measures
concentrated on $[(y_1,\dots,y_{n-1})]$ and defined by
\begin{align*}
  P\I_n([y_1\dots y_{n-1}x])
  &:=
  P\I_n([x]),\\
  P\II_n([y_1\dots y_{n-1}x])
  &:=
  P\II_n([x]),
  \quad
  x\in\mathbf{Y}^*,
\end{align*}
where the $P\I_n$ and $P\II_n$ on the right-hand sides
are the predictions in the former.

\begin{remark}{\upshape
  Jeffrey traces the idea of radical probabilism back to Ramsey and de Finetti.
  In \citet[p.~66]{Jeffrey:1992book}, he says,
  ``this is his radical probabilism---Ramsey denies
  that our probable knowledge need be based on certainties''.
  And in \citet[p.~2]{Jeffrey:1988}, he says,
  ``De Finetti's probabilism is `radical' in the sense of going all the way down to the roots:
  he sees probabilities as ultimate forms of judgment
  which need not be based on deeper all-or-none knowledge.''
  (Although there is some tension between this interpretation of de Finetti's views
  and de Finetti's emphatic defence of Bayesian conditioning
  in \citealt[Sect.~4.5.3]{deFinetti:2017}.)}
\end{remark}

\begin{remark}{\upshape
  Jeffrey often uses ``radical probabilism'' in the sense of our ``agnostic probabilism'':
  ``Radical probabilism doesn't insist that probabilities be based on certainties''
  \citep[p.~11]{Jeffrey:1992book}.
  However, in other places he appears to deny the existence of certain empirical evidence;
  e.g., in one of his earliest (1968) publications on radical probabilism he says,
  ``Radical probabilism adds [to ``probabilism''] the ``nonfoundational'' thought
  that there is no bedrock of certainty underlying our probabilistic judgments''
  \citep[pp.~44--45]{Jeffrey:1992book}.}
\end{remark}

\begin{example}\label{ex:radical}{\upshape
  Let us check that Theorem~\ref{thm:main} does not hold under agnostic probabilism
  (without any further conditions).
  Suppose $P\I_n = P\I$ and $P\II_n = P\II$ do not depend on $n$.
  Then we have $\K\I_n = \K\II_n = 1$ for all $n$,
  and so we have no convergence of opinion for successful forecasters
  even if $P\I$ and $P\II$ are very different.}
\end{example}

\begin{theorem}\label{thm:radical}
  Sceptic has a strategy in Protocol~\ref{prot:additive}
  that guarantees $\left\|P\I_n-P\II_n\right\|\to0$ as $n\to\infty$
  whenever the following three conditions are satisfied:
  \begin{itemize}
  \item
    $P\I_n-P\II_n\to0$ weakly as $n\to\infty$,
  \item
    $\K\I_n\not\to\infty$ as $n\to\infty$, and
  \item
    $\K\II_n\not\to\infty$ as $n\to\infty$.
  \end{itemize}
\end{theorem}

Theorem~\ref{thm:radical} includes Theorem~\ref{thm:main}
as a special case.
According to Example~\ref{ex:radical},
we need to impose some condition of agreement
(which we want to make as weak as possible)
between the two forecasters
before we can claim that they agree in the strong sense of convergence in total variation.
In Theorem~\ref{thm:main} both forecasters weakly agree with Reality
(and therefore, between themselves):
for a fixed $x$,
$P\I_n([x])$ and $P\II_n([x])$ of Protocol~\ref{prot:additive}
become equal (namely, both equal to 0 or to 1)
as soon as $n$ exceeds the length of $x$.
In Theorem~\ref{thm:radical} we require an even weaker agreement between forecasters
(and do not require any agreement with Reality,
who is not even a player).

In Sect.~\ref{sec:basic}
we interpreted Jeffreys's law in the form of Theorem~\ref{thm:main}
as establishing a connection between the correspondence and convergence theories of truth.
There is no absolute notion of truth under radical probabilism,
and so this interpretation is not applicable to Theorem~\ref{thm:radical}.
Theorem~\ref{thm:radical} merely provides a means of boosting agreement
between successful forecasters:
agreement in the sense of weak convergence
implies agreement in the sense of convergence in total variation.

\section{Proofs}
\label{sec:proofs}

In this section we will prove Theorems~\ref{thm:main} (in Sect.~\ref{subsec:proof-main})
and~\ref{thm:radical} (in Sect.~\ref{subsec:proof-radical}).

\subsection{Proof of Theorem~\ref{thm:main}}
\label{subsec:proof-main}

For simplicity,
in the proofs in this section we assume that $P\I_n([x])>0$ and $P\II_n([x])>0$
for all $x\in\mathbf{Y}^*$
(``Cromwell's rule'').

We will construct a strategy for Sceptic
that guarantees the disjunction of
\begin{itemize}
\item
  $\left\|P\I_n-P\II_n\right\|\to0$ as $n\to\infty$,
\item
  $\sqrt{\K\I_n\K\II_n}$ is unbounded as $n\to\infty$.
\end{itemize}
This is sufficient since we can apply Proposition~11.2 in \citet{Shafer/Vovk:2019}
to turn an unbounded capital $\K\I_n$ or $\K\II_n$
into $\K\I_n\to\infty$ or $\K\II_n\to\infty$.

As a second step,
let us replace $\left\|P\I_n-P\II_n\right\|\to0$ by $H(P\I_n,P\II_n)\to1$,
where $H$ is the \emph{Hellinger integral}
\citep[Definition~3.9.3]{Shiryaev:2016}:
\begin{equation*}
  H(P,Q)
  :=
  \int_{\mathbf{Y}^{\infty}}
  \sqrt{\dd P \dd Q}
  \in
  [0,1].
\end{equation*}
This can be done because of the standard connection
between total variation distance and Hellinger integral
\citep[Theorem~3.9.1]{Shiryaev:2016}:
\begin{equation}\label{eq:H}
  2
  \left(
    1 - H(P,Q)
  \right)
  \le
  \left\|
    P-Q
  \right\|
  \le
  \sqrt
  {
    8
    \left(
      1 - H(P,Q)
    \right)
  }.
\end{equation}

We will also need an approximation to $H(P,Q)$,
given in the following lemma,
in terms of
\[
  H_m(P,Q)
  :=
  \int_{\mathbf{Y}^m}
  \sqrt{(\dd P|_m)(\dd Q|_m)},
\]
where $P|_m$ and $Q|_m$ stand for the restrictions of $P$ and $Q$
to the first $m$ observations;
in other words, $P|_m$ and $Q|_m$ are the probability measures on $\mathbf{Y}^m$
satisfying $(P|_m)(x)=P(x)$ and $(Q|_m)(x)=Q(x)$ for all $x\in\mathbf{Y}^m$.

\begin{lemma}
  As $m\to\infty$,
  $H_m(P,Q)\to H(P,Q)$.
\end{lemma}

\begin{proof}
  Applying Doob's martingale convergence theorem
  \citep[Theorem~7.4.1]{Shiryaev:2019}
  to $(\dd P|_m) / (\dd R|_m)$ and $(\dd Q|_m) / (\dd R|_m)$,
  where $R:=(P+Q)/2$,
  we obtain their $R$-almost sure convergence as $m\to\infty$
  to $\dd P / \dd R$ and $\dd Q / \dd R$,
  respectively
  (this step uses Cara\-th\'eo\-dory's theorem,
  as in the proof of \citealt[Theorem~7.6.1]{Shiryaev:2019}).
  Therefore,
  \[
    \sqrt{\frac{\dd P|_m}{\dd R|_m} \frac{\dd Q|_m}{\dd R|_m}}
    \to
    \sqrt{\frac{\dd P}{\dd R} \frac{\dd Q}{\dd R}}
  \]
  $R$-almost surely
  and, since all these fractions take values in $[0,2]$,
  in $L_1$ w.r.\ to $R$.
\end{proof}

\begin{algorithm}[bt]
  \caption{Betting strategy for Sceptic for a fixed $\epsilon\in(0,1)$}
  \label{alg:BD}
  \begin{algorithmic}[1]
    \For{$n=1,2\dots$:}
      \State Observe $P\I_n$ and $P\II_n$ on step $n$ of Protocol~\ref{prot:basic-2}
      \If{$H(P\I_n,P\II_n)<1-\epsilon$:}\label{l:If}
        \State Find $m$ such that $H_m(P\I_n,P\II_n)<1-\epsilon$
        \State Buy a collection of $x$-tickets, $x\in\mathbf{Y}^m$, from Forecaster~I that\label{l:Buy-1}
          \Statex\hspace{\algorithmicindent*\real{3}}
            multiplies the current $\K\I_n$ by
            $\displaystyle\frac{1}{H_m(P\I_n,P\II_n)}\frac{\dd\sqrt{P\I_n|_m P\II_n|_m}}{\dd P\I_n|_m}$
        \State Buy a collection of $x$-tickets, $x\in\mathbf{Y}^m$, from Forecaster~II that\label{l:Buy-2}
          \Statex\hspace{\algorithmicindent*\real{3}}
            multiplies the current $\K\II_n$ by
            $\displaystyle\frac{1}{H_m(P\I_n,P\II_n)}\frac{\dd\sqrt{P\I_n|_m P\II_n|_m}}{\dd P\II_n|_m}$
        \State Skip the next $m$ steps\label{l:Skip}
      \EndIf
    \EndFor
  \end{algorithmic}
\end{algorithm}

Fix temporarily an $\epsilon>0$ (we will later mix over a sequence of $\epsilon\to0$).
Given $\epsilon$, Sceptic can bet at the steps $n=1,2,\dots$ as in Algorithm~\ref{alg:BD}.
The instructions in Algorithm~\ref{alg:BD} should be read in parallel
with the description below.

In our interpretation of Protocol~\ref{prot:basic-2}
we assumed that at each step Sceptic sells all tickets
bought at the previous step and buys new tickets.
For a given $x$-ticket,
an important special case is where the new amount of $x$-tickets
is equal to the old amount
(and so we can assume that no trade in $x$-tickets takes place at this step).
In particular, Sceptic can buy an $x$-ticket at any step $n$
and hold it to maturity collecting $1_{\{(y_n,\dots,y_{n+m-1})=x\}}$ at step $n+m$,
where $m:=\lvert x\rvert$.

Line~\ref{l:Buy-1} of Algorithm~\ref{alg:BD} instructs Sceptic to buy a collection of $x$-tickets
multiplying his current capital~by
\begin{multline}\label{eq:goal}
  \frac{1}{H_m(P\I_n,P\II_n)}
  \frac{\dd\sqrt{P\I_n|_m P\II_n|_m}}{\dd P\I_n|_m}\\
  =
  \frac{1}{H_m(P\I_n,P\II_n)}
  \frac
  {\sqrt{P\I_n([(y_n,\dots,y_{n+m-1})]) P\II_n([(y_n,\dots,y_{n+m-1})])}}
  {P\I_n([(y_n,\dots,y_{n+m-1})])}.
\end{multline}
Let us check that it is indeed possible
to turn an initial capital of 1 into \eqref{eq:goal}.
By definition, Sceptic can buy from Forecaster~I a ticket paying $1_{\{x=(y_n,\dots,y_{n+m-1})\}}$
(i.e., the $x$-ticket)
for $P\I_n(x)$,
for any $x\in\mathbf{Y}^m$.
Therefore, Sceptic can buy a ticket paying
\[
  \frac
  {\sqrt{P\I_n([x]) P\II_n([x])}}
  {P\I_n([x])}
  1_{\{x=(y_n,\dots,y_{n+m-1})\}}
\]
for $\sqrt{P\I_n([x]) P\II_n([x])}$, for any $x\in\mathbf{Y}^m$.
Buying such a ticket for each $x\in\mathbf{Y}^m$ results in a payoff of
\begin{equation*}
  \frac
  {\sqrt{P\I_n([(y_n,\dots,y_{n+m-1})]) P\II_n([(y_n,\dots,y_{n+m-1})])}}
  {P\I_n([(y_n,\dots,y_{n+m-1})])}
\end{equation*}
for the price of
\[
  \sum_{x\in\mathbf{Y}^m}
  \sqrt{P\I_n([x]) P\II_n([x])}
  =
  H_m(P\I_n,P\II_n).
\]
Therefore, we can indeed turn 1 into \eqref{eq:goal}.

A similar argument applies to Forecaster~II
(line~\ref{l:Buy-2} of Algorithm~\ref{alg:BD}).

When the condition in line~\ref{l:If} is satisfied,
the geometric mean $\sqrt{\K\I_n\K\II_n}$ of Sceptic's capitals
over the $m$ steps in line~\ref{l:Skip}
will be multiplied by $1/H_m(P\I_n,P\II_n)$,
i.e., by more than $\frac{1}{1-\epsilon}$.
Therefore,
$\sqrt{\K\I_n\K\II_n}$ will be unbounded
if the condition in line~\ref{l:If} is satisfied infinitely often.

Let $\K\I_n(\epsilon)$ and $\K\II_n(\epsilon)$ be the capital processes
that result from using Algorithm~\ref{alg:BD} with parameter $\epsilon$
in Protocol~\ref{prot:basic-2}.
Then
\begin{equation*}
  \K\I_n := \sum_{j=1}^{\infty} 2^{-j} \K\I_n(2^{-j})
  \quad\text{and}\quad
  \K\II_n := \sum_{j=1}^{\infty} 2^{-j} \K\II_n(2^{-j})
\end{equation*}
are also capital processes for a valid strategy for Sceptic,
and $\sqrt{\K\I_n\K\II_n}$ will be unbounded
unless $H(P\I_n,P\II_n)\to1$.
This completes the proof.

\begin{remark}{\upshape
  The proof given in this section relies
  on the method of combining probability measures $P\I$ and $P\II$
  known as ``geometric pooling'';
  see, e.g., \citet{Pettigrew/Weisberg:2024} for a recent review.}
\end{remark}

\subsection{Proof of Theorem~\ref{thm:radical}}
\label{subsec:proof-radical}

We follow the proof of Theorem~\ref{thm:main} in Sect.~\ref{subsec:proof-main}
modifying it slightly.
We still use Algorithm~\ref{alg:BD},
but now $P\I_n$ and $P\II_n$ have a different meaning:
both of them are predictions
for the whole sequence of observations $y_1,y_2,\dots$
rather than only for the future observations $y_n,y_{n+1},\dots$
(there is no clear-cut division between past and future under radical probabilism).

It will be convenient to say that
in~\eqref{eq:K-I} in Protocol~\ref{prot:additive}
Sceptic invests in the $x$-tickets at step $n-1$,
with his initial investment being $P\I_{n-1}([x])$
and his payoff being $P\I_n([x])$ for each $x$-ticket.
Since Sceptic can hold the same position $f\I_n(x)$ over a number of steps,
he can invest in the $x$-tickets at step $n$
receiving a payoff at step $N>n$,
with the initial investment being $P\I_{n}([x])$
and the final payoff being $P\I_N([x])$.

Now we modify the argument leading to the possibility
to increase Sceptic's capital \eqref{eq:goal}-fold
(where \eqref{eq:goal} is the equation number).
Sceptic can invest $P\I_n(x)$ in the $x$-ticket at step $n$,
and his payoff at step $N>n$ will be $P\I_N(x)$,
for any $x\in\mathbf{Y}^m$.
Therefore, Sceptic can invest $\sqrt{P\I_n([x]) P\II_n([x])}$
for a payoff of
\[
  \frac
    {\sqrt{P\I_n([x]) P\II_n([x])}}
    {P\I_n([x])}
  P\I_N(x).
\]
Investing in each $x\in\mathbf{Y}^m$ results in a payoff of
\begin{equation*}
  \sum_{x\in\mathbf{Y}^m}
  \frac
    {\sqrt{P\I_n([x]) P\II_n([x])}}
    {P\I_n([x])}
  P\I_N(x)
\end{equation*}
for the initial investment of
\[
  \sum_{x\in\mathbf{Y}^m}
  \sqrt{P\I_n([x]) P\II_n([x])}
  =
  H_m(P\I_n,P\II_n).
\]
Therefore,
starting from step~\ref{l:Buy-1} of Algorithm~\ref{alg:BD}
the capital $\K\I_n$ can be multiplied by
\begin{equation}\label{eq:increase-1}
  \frac{1}{H_m(P\I_n,P\II_n)}
  \sum_{x\in\mathbf{Y}^m}
  \frac{\sqrt{P\I_n([x])P\II_n([x])}}{P\I_n([x])}
  P\I_N([x])
\end{equation}
by step $N$
(later we will make $N$ large, definitely $N>m$),
and starting from step~\ref{l:Buy-2}
the capital $\K\II_n$ will be multiplied by
\begin{equation}\label{eq:increase-2}
  \frac{1}{H_m(P\I_n,P\II_n)}
  \sum_{x\in\mathbf{Y}^m}
  \frac{\sqrt{P\I_n([x])P\II_n([x])}}{P\II_n([x])}
  P\II_N([x])
\end{equation}
by step $N$.

Using the weak convergence $P\I_N-P\II_N\to0$
and the uniform continuity of \eqref{eq:increase-1} as function of $P\I_N|_m$
and of \eqref{eq:increase-2} as function of $P\II_N|_m$,
we can replace, asymptotically, $P\I_N|_m$ and $P\II_N|_m$ by the same probability measure $P$ on $\mathbf{Y}^m$,
and so the geometric mean of \eqref{eq:increase-1} and \eqref{eq:increase-2}
can be bounded below, for a sufficiently large $N$, by the last term of the chain
\begin{align*}
  &\frac{1}{H_m(P\I_n,P\II_n)}
  \sqrt
  {
    \left(
      \sum_{x\in\mathbf{Y}^m}
      \frac{\sqrt{P\I_n([x])P\II_n([x])}}{P\I_n([x])}
      P([x])
    \right)
    \left(
      \sum_{x\in\mathbf{Y}^m}
      \frac{\sqrt{P\I_n([x])P\II_n([x])}}{P\II_n([x])}
      P([x])
    \right)
  }\\
  &\ge
  \frac{1}{H_m(P\I_n,P\II_n)}
  \sum_{x\in\mathbf{Y}^m}
  \sqrt
  {
    \frac{\sqrt{P\I_n([x])P\II_n([x])}}{P\I_n([x])}
    \frac{\sqrt{P\I_n([x])P\II_n([x])}}{P\II_n([x])}
  }
  P([x])\\
  &=
  \frac{1}{H_m(P\I_n,P\II_n)}
  >
  \frac{1}{1-\epsilon},
\end{align*}
where the ``$\ge$'' follows from the concavity of the geometric mean function
$(u,v)\in[0,\infty)^2\mapsto\sqrt{u v}$
and Jensen's inequality \citep[Lemma 2.8.1]{Ferguson:1967}.
Therefore, we can wait until the geometric mean of the capitals increases
$\frac{1}{1-\epsilon}$-fold.
The proof is completed as before,
by mixing over $\epsilon$.
Notice that in our construction Sceptic's capital cannot become negative.

\section{Comparison with known results}
\label{sec:comparison}

In this section we will discuss two kinds of Jeffreys's law:
for predicting the infinite future (Sect.~\ref{subsec:BD})
and for one-step-ahead prediction (Sect.~\ref{subsec:1-step}).

\subsection{Blackwell--Dubins result}
\label{subsec:BD}

The main topic of this subsection is limitations of Blackwell and Dubins's classical result
(\citeyear{Blackwell/Dubins:1962}, Sect.~2)
and how they are overcome by this paper's results.

From the purely mathematical point of view,
one limitation of Blackwell and Dubins's result
is that they consider two probability measures, $Q$ and $P$
(which correspond to our Forecasters~I and~II)
such that $Q$ is absolutely continuous w.r.\ to $P$,
denoted $Q\ll P$.
This requirement means that $Q(E)=0$
for any event $E$ such that $P(E)=0$.
Let us write $P\I$ and $P\II$ for $Q$ and $P$, respectively.

We regard $P\I$ as a base forecasting strategy;
the corresponding $n$th forecast $P\I_n$
is the conditional probability of the future observations $y_n,y_{n+1},\dots$
given the past observations $y_1,\dots,y_{n-1}$,
with the observations $y_1,y_2,\dots$ generated from $P\I$.
Then $P\II$ is an alternative forecasting strategy
producing, in a similar manner, $P\II_1,P\II_2,\dots$.
The condition $P\I\ll P\II$ can be interpreted as $P\II$ being at least as adaptive as $P\I$.
Blackwell and Dubins's result says that $P\I_n$ and $P\II_n$ converge in total variation
with $P\I$-probability 1 whenever $P\I$ agrees with $P\II$,
in the sense of $P\II$ being at least as adaptive as $P\I$.
(This ignores technical issues surrounding the existence of conditional distributions,
which are less acute in our current context of a finite observation space $\mathbf{Y}$.)

What if the two forecasting strategies do not agree
(we have neither $P\I\ll P\II$ nor $P\II\ll P\I$)?
For example, suppose that they agree perfectly everywhere apart from two events,
$E\I$ and $E\II$,
for which
\[
  P\I(E\I)
  =
  P\II(E\II)
  =
  10^{-6},
  \quad
  P\II(E\I)
  =
  P\I(E\II)
  =
  0.
\]
For this case the Blackwell--Dubins result does not say anything.
To prevent the possibility of waiting until one or both events become settled,
suppose further that, for all $n$,
\begin{multline*}
  P\I(E\I\mid y_1,\dots,y_n),
  P\II(E\I\mid y_1,\dots,y_n),\\
  P\I(E\II\mid y_1,\dots,y_n),
  P\II(E\II\mid y_1,\dots,y_n)
  \in
  (0,1).
\end{multline*}

In our context, the condition of absolute continuity becomes not only restrictive but also less natural.
Namely, in our framework, there are no \emph{a priori} connections
between the probability forecasts $P\I_1,P\I_2,\dots$ and $P\II_1,P\II_2,\dots$
output at different steps,
and so it is possible to have $P\I_1\ll P\II_1$
followed by $P\I_2\mathrel{\bot}P\II_2$ ($\bot$ meaning mutual singularity)
followed by $P\I_3\ll P\II_3$, etc.

The condition of absolute continuity is very natural (or even unavoidable)
under the Bayesian interpretation of Blackwell and Dubins's result
(both Blackwell and Dubins were Bayesians \citep[pp.~43--44 and p.~48]{DeGroot:1986}).
Its Bayesian interpretation is that Forecaster~I
believes that the forecasts issued by the two forecasters
will converge in total variation.
Bayesian theory is based on personal probability,
and for a Bayesian interpretation it is essential
to indicate the person whose beliefs we are talking about.
It is difficult to see who can be such a person
without the requirement of absolute continuity.

Our interpretation of Theorem~\ref{thm:radical}
was in terms of boosting:
weak convergence can be boosted to convergence in total probability.
Blackwell and Dubins's result can also be interpreted in these terms:
agreement between two forecasting strategies in the sense of $P\I\ll P\II$
can be boosted to convergence in total variation $P\I$-almost surely.

The game-theoretic version of Blackwell and Dubins's result
proved in Sect.~\ref{sec:basic} of this paper
not only removes some of the unnecessary restrictions
but is also more constructive.
We have an explicit strategy for Sceptic that discredits at least one of the forecasters
by successful betting against his forecasts
unless the predictions that the forecasters output converge in total variation.
While Blackwell and Dubins's result under its Bayesian interpretation
only concerns a Bayesian's beliefs,
our result establishes connections with idealized reality.

\subsection{One-step-ahead prediction}
\label{subsec:1-step}

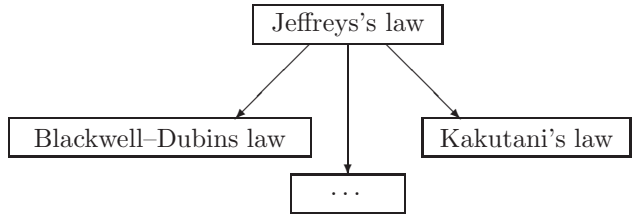
\begin{figure}
  \setlength{\unitlength}{0.5cm}
  \begin{center}
    \begin{picture}(17,5.5)(-9,-3.5)
      \put(-1,1){\vector(-1,-1){2}}
      \put(1,1){\vector(1,-1){2}}
      \put(0,1){\vector(0,-1){3.5}}
      \put(-2.5,1){\framebox(5,1)[cc]{Jeffreys's law}}
      \put(-9,-2){\framebox(8,1)[cc]{Blackwell--Dubins law}}
      \put(2,-2){\framebox(5.5,1)[cc]{Kakutani's law}}
      \put(-1.5,-3.5){\framebox(3,1)[cc]{\dots}}
    \end{picture}
  \end{center}
  \caption{Two known special cases of Jeffreys's law}
  \label{fig:Jeffreys}
\end{figure}

Figure~\ref{fig:Jeffreys} represents
two known special cases of Jeffreys's law.
Blackwell and Dubins's result is an instance of one of them;
more generally, Fig.~\ref{fig:Jeffreys} refers to the \emph{Blackwell--Dubins law}
as an asymptotic qualitative statement about convergence
between forecasts for the infinite future
(or, in the case of radical probabilism, infinite past, present, and future).
Theorem~\ref{thm:main} and its generalization Theorem~\ref{thm:radical}
are also results of this kind.
Despite overcoming some limitations of Blackwell and Dubins's result,
they are still asymptotic and qualitative.

A very different approach to Jeffreys's law was pioneered by \citet{Kakutani:1948},
with a crucial step made by \citet{Kabanov/etal:1977}.
Their results also show that, provided $P\I\ll P\II$,
the one-step-ahead predictions computed from $P\I$ and $P\II$ converge,
but the nature of these results very different.
Their important advantage is that they are quantitative
(albeit with interesting qualitative implications);
the price to pay, however, is that they only cover one-step ahead prediction.

The results by \citet{Kakutani:1948} and~\citet{Kabanov/etal:1977}
are measure-theoretic and have similar disadvantages
to those of the Blackwell--Dubins result
discussed in Sect.~\ref{subsec:BD},
but these disadvantages were eliminated in the game-theoretic versions
described in \citet[Sect.~10.7]{Shafer/Vovk:2019}
(and developed in the references given in \citealt[Sect.~10.9]{Shafer/Vovk:2019}).
One such result is that in Protocol~\ref{prot:basic-2} Sceptic can ensure
\begin{equation}\label{eq:quantitative}
  \ln\K\I_n
  +
  \ln\K\II_n
  \ge
  \frac14
  \sum_{i=1}^n
  \left\|
    (P\I_i|_1) - (P\II_i|_1)
  \right\|^2
\end{equation}
for all $n$,
where $P\I_i|_1$ and $P\II_i|_1$ are the restrictions of $P\I_i$ and $P\II_i$,
respectively, to $\mathbf{Y}$, as defined earlier
(i.e., they are the one-step-ahead restrictions).
(To obtain \eqref{eq:quantitative}, combine Proposition 10.17 in \citet{Shafer/Vovk:2019}
with the standard bound for Hellinger distance,
which is essentially the right-hand side of \eqref{eq:H}
\citep[Theorem 3.9.1, (23)]{Shiryaev:2016}.)

It is clear that \eqref{eq:quantitative} implies Theorem~\ref{thm:main}
with  $P\I_n$ and $P\II_n$ replaced by $P\I_n|_1$ and $P\II_n|_1$,
respectively.
However, \eqref{eq:quantitative} is an explicit lower bound rather than merely an asymptotic result.
Figure~\ref{fig:Jeffreys} refers to the class of such quantitative one-step-ahead results
as \emph{Kakutani's law}.
The ellipsis represents new special cases of Jeffreys's law,
those yet to be discovered.

\section{Conclusion}
\label{sec:conclusion}

A big disadvantage of Theorems~\ref{thm:main} and~\ref{thm:radical},
inherited from their measure-theoretic prototype \citep{Blackwell/Dubins:1962},
is that they are merely asymptotic;
it is not obvious how to make them quantitative.
The versions given in \citet[Sect.~10.7]{Shafer/Vovk:2019}
are much more precise, but they cover only one-step-ahead forecasting.
Bridging the gap between these two very different kinds of results
(depicted in Fig.~\ref{fig:Jeffreys})
looks to me an interesting direction of further research.

Other possible directions of further research are:
\begin{itemize}
\item
  Generalizing our testing protocols and results
  to the case of an infinite observation space $\mathbf{Y}$.
\item
  Are there situations where a wider set of permitted moves for Sceptic would be useful?
  In Protocols~\ref{prot:basic-1}--\ref{prot:basic-2} and~\ref{prot:additive}
  we considered Sceptic's moves $f$ such that $f(x)=0$ apart from finitely many $x$.
  More generally, we could permit $f$ such that $f(x)\ne0$ for arbitrarily long $x$,
  but this would require careful analysis of convergence of the resulting infinite series
  in expressions for Sceptic's capital.
\end{itemize}

\subsection*{Acknowledgments}

Many thanks to A. Philip Dawid for his questions and comments.
Research on this paper has been partially supported by Mitie.

\end{document}